\documentclass[journal]{IEEEtran}

\usepackage{amsmath,amsthm}
\usepackage{amssymb}
\usepackage{mathrsfs}
\usepackage{amsfonts}
\usepackage{graphics}
\usepackage{float}
\usepackage{graphicx}
\usepackage{indentfirst}
\usepackage{enumerate}
\usepackage{caption}
\usepackage{cite}
\usepackage{arydshln}
\usepackage{bm}

\theoremstyle{definition}
\newtheorem{theorem}{Theorem}
\newtheorem{lemma}{Lemma}
\newtheorem{corollary}{Corollary}
\newtheorem{definition}{Definition}

\newtheorem{remark}{Remark}

\newtheorem{example}{Example}

\def\Ds{\displaystyle}

\makeatletter

\newcommand{\Rmnum}[1]{\expandafter\@slowromancap\romannumeral #1@}
\makeatother
\ifCLASSINFOpdf
\else
\fi

\begin{document}

\title{Characterization of Collective Behaviors for Directed Signed Networks}
\author{Wen~Du,~Deyuan~Meng,~\IEEEmembership{Senior~Member,~IEEE},~and~Mingjun~Du

\thanks{Wen Du is with Department of Electrical Engineering, University of North Texas, Denton, TX 76201, USA (Email: wendu@my.unt.edu).}
\thanks{Deyuan Meng (corresponding author) is with the Seventh Research Division, Beihang University (BUAA), Beijing 100191, P. R. China, and also with the School of Automation Science and Electrical Engineering, Beihang University (BUAA), Beijing 100191, P. R. China (Email: dymeng@buaa.edu.cn).}
\thanks{Mingjun Du is with the School of Electrical Engineering and Automation, Qilu University of Technology (Shandong Academy of Science), Jinan, Shandong Province 250353, P. R. China. (Email: dumingjun0421@163.com).}}

\maketitle

\begin{abstract}
This paper targets at exploring how to charactering collective behaviors of directed signed networks. 
The right eigenvector of the Laplacian matrix associated with zero eigenvalue is further investigated and its mathematical expression is proposed. 
It is shown that the right eigenvector plays an important role in determining the collective behaviors of directed signed networks.  
Furthermore,  algebraic criteria are introduced for collective behaviors of directed signed networks, such as bipartite consensus, interval bipartite consensus and bipartite containment tracking. 
In addition, a simulation example is given to the correctness of our developed theoretical results.

\end{abstract}


\IEEEpeerreviewmaketitle

\section{Introduction}

The distributed coordination of unsigned networks has received considerable attentions because of its wide applications in many fields, such as multi-mobile robots, unmanned air vehicles (UAVs), autonomous underwater vehicles (AUVs) (see \cite{Lin04} for more details). 
Among those applications, the cooperative control has been widely used for distributed coordination of unsigned networks. 
The consensus problem is one of the most fundamental problems in the studies of the cooperative control.
It is defined as  that a group of agents cooperating with each other  achieve a common decision by exploiting the information  collected from their neighbors.

Fruitful results of consensus problems have been developed in the past decades \cite{Cao08,Yang12,Yang14,Char15}. 
For unsigned networks with first-order integrator dynamics,  communication topologies play an important role in analyzing consensus problems, where the strong connectivity \cite{Olfa04} and quasi-strong connectivity (or containing a spanning tree) \cite{Ren05}  ensure the consensus. 
In \cite{Li09}, the algebraic criteria have been given for identifying the communication topologies of unsigned networks, in which the left eigenvector of the Laplacian matrix associated with the zero eigenvalue has been employed.
 Besides, leader-follower consensus needs to be mentioned, that is, if there exists only one leader in the network,  each follower tracks the trajectory of leader's \cite{Hu07,Ni10,Su13}. 
However, leader-follower consensus fails when there exist multiple leaders in the unsigned networks. 
Toward this problem, the containment tracking control  has been introduced in \cite{Ji08,Cao12,Li16}, in which multiple leaders are allowed and followers converge to the convex hull spanned by leaders.

All of these results are under the assumption that all agents of networks cooperate with each other. However, there  exist many networked systems in the real world, such as market, biological systems and social networks, that does not always consist of cooperating agents. 
Antagonistic connections (such as contraversy, disagreement, dislike and conflict) among agents also exist\cite{Easl12,Facc11,Wass94}. These kinds of networked systems are named as signed networks. 
Signed digraphs associated with positive and negative adjacency weights are proposed to depict the  signed networks with cooperative and antagonistic connections among agents.

In signed networks, the traditional framework of exploring consensus problems is not suitable anymore because of the existence of antagonistic connections. 
Instead, bipartite consensus, which is defined as that  final values of  agent states are the same in module but opposite in sign, is considered when polarization phenomena exist in signed networks. 
Bipartite consensus problems have been investigated in \cite{Alta13}. 
It is shown that for signed networks under strongly connected signed digraphs, bipartite consensus  can be achieved if the signed digraph is structurally balanced, and the state stability is achieved, otherwise. 
Motivated by the results of \cite{Alta13},  bipartite consensus problems have  been studied for signed networks subject to general linear dynamics \cite{Valc14,Qin17,Jiang17}, with communication noises \cite{Hu19} and under dynamic topologies \cite{Meng18}. 
Besides, in \cite{Wang18,Lu20},  distributed control protocols have been designed to ensure the finite-time bipartite consensus of signed networks. 
When the topology of signed networks  is quasi-strongly connected signed digraphs, the notation of interval bipartite consensus has been proposed  \cite{Meng16,Xia16,Meng20}, which indicates that all rooted-agents  achieve the bipartite consensus and non-rooted agents  spread in the interval constructed by two convergency values of rooted-agents. 
For signed networks under any signed digraphs, the bipartite containment tracking has been given  \cite{Meng17}. 
It means that followers converge to the convex hull determined by each leader's trajectory as well as its symmetric trajectory which is the same in modulus but different in sign.

From \cite{Olfa04,Li09,Du19},  the terminal value of all agents is determined by initial states, left eigenvector and right eigenvector that are corresponding to the zero eigenvalue of Laplacian matrix. 
The mathematical expression for the left eigenvector has been provided for the unsigned digraph and the signed digraph in \cite{Li09} and \cite{Du19}, respectively. 
When considering the strongly connected and quasi-strongly connected unsigned digraphs, all entries of the right eigenvector are equal. 
For the strongly connected and structurally balanced signed digraphs, the moduli of all entries of the right eigenvector are the same. 
However, when considering quasi-strongly connected signed digraphs, there  exist various cases for the right eigenvector. 
Until now, the mathematical expression has not been proposed for the right eigenvector of Laplacian matrix associated with the zero eigenvalue.

Motivated by  discussions above, this paper concentrates on providing the mathematical expression for the right eigenvector of Laplacian matrices corresponding to the zero eigenvalue and further deriving the algebraic criteria for identifying the collective behaviors of signed networks. 
Firstly, based on the Cramer's Rule,   the mathematical expression of the right eigenvector is given. 
Secondly, for quasi-strongly connected signed digraphs,   the relationship between the right eigenvector and the structurally balanced (or unbalanced)  property of a node is revealed. 
Namely, a node is a structurally balanced (respectively, unbalanced) node if and only if the absolute value of its corresponding entry in the right eigenvector is  equal to (respectively, is less than) one. 
Thirdly,  algebraic criteria for identifying  collective behaviors  achieved by signed networks are proposed. 
Finally, one simulation example is provided to demonstrate the effectiveness of  theoretical results.

The rest of this paper is organized as follows. 
In Section \Rmnum{2},  notations and basic knowledge for signed digraphs are introduced. 
In Section \Rmnum{3}, the problem formulation is provided. 
In Section \Rmnum{4}, the mathematical expression for the right eigenvector of Laplacian matrix associated with the zero eigenvalue is proposed. 
In Section \Rmnum{5}, sufficient and necessary conditions based on different forms of right eigenvectors are given for different collective behaviors. 
Besides, the relationship between structurally balanced property of a node and the right eigenvector is revealed.
A numerical example and conclusions are given in Sections \Rmnum{6} and \Rmnum{7}, respectively.

\section{notations and preliminaries}\label{notations and preliminaries}

\subsection{Notations}
For a positive integer $n$, we denote $\mathcal{I}_n=\{1, 2, \cdots, n\}$, $\mathbf{1}_n=[1,1, \cdots, 1]^{\top}\in \mathbb{R}^{n}$, $\mathbf{0}_{n}=[0,0,\cdots,0]^{\top}\in \mathbb{R}^{n}$ and $\mathrm{diag}\{d_1, d_2, \cdots, d_n\}$ as diagonal matrix whose diagonal elements are $d_1, d_2, \cdots, d_n$. 
For a square matrix $M\in \mathbb{R}^{n\times n}$, let $\mathrm{det}(M)$, $\mathcal{N}(M)$ and $M^{*}$ represent the determinant, null space and adjoint matrix of $M$, respectively. 
For a real number $a \in \mathbb{R}$, let $|a|$ and $\mathrm{sgn}(a)$ represent the absolute value and the sign function of $a$, respectively. 
The set of all $n$-by-$n$ gauge transformations is given by
\begin{equation*}
\mathcal{D}_n=\{D_n=\mathrm{diag}\{\sigma_1, \sigma_2, \dots, \sigma_n\}: \sigma_i \in \{-1, 1\}, i \in \mathcal{I}_n\}.
\end{equation*}

\subsection{Signed Digraphs}

A weighted signed digraph is represented by a triple $\mathcal{G}=(\mathcal{V}, \mathcal{E}, A)$, where $\mathcal{V}=\{v_1, v_2, \dots, v_n\}$ represents its node set, $\mathcal{E}\subseteq \mathcal{V}\times \mathcal{V}=\{(v_j, v_i): v_i, v_j \in\mathcal{V}\}$ is its edge set, which is defined such that $(v_j, v_i)$ is a directed edge from $v_j$ to $v_i$ when the node $v_j$ is a neighbor of node $v_i$, and $\mathcal{A}=[a_{ij}]\in\mathbb{R}^{n\times n}$ is its adjacency weight matrix which is defined such that $(v_j, v_i)\in \mathcal{E}\Leftrightarrow a_{ij}\neq 0$ and otherwise $a_{ij}=0$. Assume that $\mathcal{G}$ has no self-loops, i.e., $a_{ii}=0, \forall i \in \mathcal{I}_n$. Let $\mathcal{N}_i = \{j: (v_j, v_i)\in\mathcal{E}\}$ denote the set of labels of those nodes that are neighbors of $v_i$. The Laplacian matrix $L$ of $\mathcal{G}$ is defined as  $L=[l_{ij}]\in \mathbb{R}^{n \times n}$ whose elements satisfy $l_{ij}=\sum_{k\in \mathcal{N}_i}|a_{ik}|$ if $j=i$, and $l_{ij}=-a_{ij}$, otherwise. The signed graph $\mathcal{G}$ is structurally balanced if it admits a bipartition of the nodes $\mathcal{V}_1$, $\mathcal{V}_2$, $\mathcal{V}_1\cup\mathcal{V}_2=\mathcal{V}$, $\mathcal{V}_1\cap\mathcal{V}_2=\emptyset$, such that $a_{ij}\geq 0 \ \forall v_i, v_j\in \mathcal{V}_q(q\in \{1, 2\}), a_{ij}\leq 0 \ \forall v_i\in\mathcal{V}_q, v_j\in\mathcal{V}_r, q \neq r(q, r\in \{1, 2\})$. Otherwise, the signed graph $\mathcal{G}$ is structurally unbalanced.

A directed path $\mathcal{P}$ (of length $l$) in $\mathcal{G}$ from the initial node $v_i$ to the terminal node $v_j$ is constructed by a finite sequence of edges: $\mathcal{P}=(v_{k_0}, v_{k_1})$, $(v_{k_1}, v_{k_2})$, $\dots$, $(v_{k_{l-1}}, v_{k_l})$, where $k_0 = i$, $k_l=j$ and $v_{k_0}, v_{k_1}, \dots, v_{k_l}$ are distinct nodes. 
The initial node $v_i$ is also said to be a parent node of the terminal node $v_j$. 
All parent nodes of $v_j$ are denoted as $\mathcal{V}_j^{o}$. 
A directed cycle $\mathcal{C}$ (of length $l$) in $\mathcal{G}$ is a closed directed path whose initial node and terminal node are same. Note that, since there are no self-loops in $\mathcal{G}$, we have $l \geq 1$ for a path and $l \geq 2$ for a directed cycle, respectively. 
In addition, we say that the directed cycle $\mathcal{P}$ is positive if $a_{k_1k_0}a_{k_2k_1}\cdots a_{k_0k_{l-1}}>0$ and is negative if $a_{k_1k_0}a_{k_2k_1}\cdots a_{k_0k_{l-1}}<0$. 
If there exists a directed path from the node $v_i$ to every other node in $\mathcal{G}$, then the signed digraph $\mathcal{G}$ is quasi-strongly connected (or containing a spanning tree). 
Besides, the node $v_i$ is called a rooted node.
The signed digraph $\mathcal{G}$ is strongly connected if its all nodes are rooted nodes. 
 A directed cycle is called a rooted cycle if its all nodes are rooted nodes. 
If $A=A^{\top}$ holds, then the signed digraph $\mathcal{G}$ is said to be a signed undirected graph, in which the strong connectivity collapses into connectivity.

When considering a signed digraph $\mathcal{G}_s=(\mathcal{V}_s,\mathcal{E}_s,A_s)$, we say that $\mathcal{G}_s$ is a subgraph of $\mathcal{G}=(\mathcal{V},\mathcal{E},A)$ if $\mathcal{V}_s\subseteq \mathcal{V}$ and $\mathcal{E}_s\subseteq \mathcal{E}$ hold. 
Suppose that a quasi-strongly connected signed digraph $\mathcal{G}$ has $m$ rooted nodes and $n-m$ non-rooted nodes. 
Without loss of generality,  denote $\mathcal{V}_\mathrm{r}=\{v_1,v_2,\cdots,v_m\}$ and $\mathcal{V}_{\mathrm{nr}}=\{v_{m+1},v_{m+2},\cdots,v_n\}$ as the rooted node set and the non-rooted node set of the signed digraph $\mathcal{G}$, respectively. 
We  induce two subgraphs $\mathcal{G}_\mathrm{r}=(\mathcal{V}_\mathrm{r},\mathcal{E}_\mathrm{r},A_\mathrm{r})$ and $\mathcal{G}_{\mathrm{nr}}=$ $(\mathcal{V}_{\mathrm{nr}}$, $\mathcal{E}_{\mathrm{nr}}$, $A_{\mathrm{nr}})$, where $\mathcal{E}_\mathrm{r}=\{(v_i, v_j)\in\mathcal{E}: v_i, v_j\in\mathcal{V}_\mathrm{r}\}$, $A_\mathrm{r}=[a_{ij}^\mathrm{r}]\in\mathbb{R}^{m\times m}$, $\mathcal{E}_\mathrm{nr}=\{(v_i, v_j)\in\mathcal{E}: v_i, v_j\in\mathcal{V}_\mathrm{nr}\}$, and $A_\mathrm{nr}=[a_{ij}^{\mathrm{nr}}]\in\mathbb{R}^{(n-m)\times (n-m)}$. 
 Hence, the adjacency weight matrix $A$ can be written as
\begin{align*}
A=\begin{bmatrix}
A_\mathrm{r} & 0\\
A_{\mathrm{rnr}} & A_{\mathrm{nr}}
\end{bmatrix}
\end{align*}

\noindent where $A_{\mathrm{rnr}}=[a_{ij}^{\mathrm{rnr}}]\in \mathbb{R}^{(n-m)\times m}$ with $a_{ij}^{\mathrm{rnr}}=a_{(i+m),j}, \forall 1\leq i \leq n-m, \forall 1\leq j\leq m$. The corresponding Laplacian matrix is given by
\begin{equation}\label{eq1}
L=\begin{bmatrix}
L_\mathrm{r} & 0\\
-A_{\mathrm{rnr}} & L_{\mathrm{nr}}+B
\end{bmatrix}
\end{equation}

\noindent where $L_\mathrm{r}$ and $L_{\mathrm{nr}}$ are the Laplacian matrices of the subgraphs $\mathcal{G}_\mathrm{r}$ and $\mathcal{G}_{\mathrm{nr}}$, respectively, and $B=\mathrm{diag}\{b_1,b_2,\cdots,b_{n-m}\}$ is a diagonal matrix whose element is $b_i=\sum_{j=1}^m|a_{(i+m)j}|, \forall 1\leq i\leq n-m$.

For the Laplacian matrix $L$, its eigenvalue distribution can be introduced in the following lemma.

\begin{lemma}\cite{Meng16}\label{lem1}
Consider a quasi-strongly connected signed digraph $\mathcal{G}$. Then, $L$ satisfies one of the following two results.
\begin{itemize}
  \item[R1)] $L$ has one zero eigenvalue and $n-1$ eigenvalues with positive real parts.
  \item[R2)] All eigenvalues of $L$ have positive real parts.
\end{itemize}

\noindent To be specific, the Laplacian matrix $L$ satisfies R1) if and only if one of the following two conditions holds:
\begin{enumerate}
  \item[C1)] $\mathcal{G}$ has no rooted cycles (that is, $\mathcal{G}$ has exactly one root);
  \item[C2)] all rooted cycles of $\mathcal{G}$ are positive.
\end{enumerate}

\noindent In addition, the Laplacian matrix $L$ satisfies R2) if and only if the condition C3) holds:
\begin{enumerate}
  \item[C3)] at least one rooted cycle of $\mathcal{G}$ is negative.
\end{enumerate}
\end{lemma}

For any signed digraph $\mathcal{G}$, the node $v_i$, $\forall i\in \mathcal{I}_n$ is called a structurally balanced node if the subgraph $\hat{\mathcal{G}}_i^{o}=\{\hat{\mathcal{V}}_i^{o},\hat{\mathcal{E}}_i^{o},\hat{A}_i^{o}\}$ is structurally balanced, where $\hat{\mathcal{V}}_i^{o}=\mathcal{V}_i^{o}\cup v_i$. Otherwise, the node $v_i$ is called a structurally unbalanced node. The following lemma  discloses the relationship between structurally balanced properties and structurally balanced (or unbalanced) nodes.

\begin{lemma}{\cite{Meng20}}
For any signed digraph $\mathcal{G}$, the following two results hold.
\begin{enumerate}
  \item [1)] The signed digraph $\mathcal{G}$ is structurally balanced if and only if all nodes are structurally balanced nodes.
  \item [2)]  The signed digraph $\mathcal{G}$ is structurally unbalanced if and only if at least one node is a structurally unbalanced node.
\end{enumerate}
\end{lemma}

When considering arbitrary signed digraph, the corresponding unsigned digraph can be given in the following definition.

\begin{definition}{\cite{Du19}}\label{def1}
Given any signed digraph $\mathcal{G}=\left(\mathcal{V},\mathcal{E},A\right)$, a digraph $\overline{\mathcal{G}}=\left(\mathcal{V},\mathcal{E},\overline{A}\right)$ is called an induced unsigned digraph of $\mathcal{G}$ if its weight matrix $\overline{A}=\left[\overline{a}_{ij}\right]\in\mathbb{R}^{n\times n}$ is defined with entries satisfying $\overline{a}_{ij}=\left|a_{ij}\right|$, $\forall i,j\in\mathcal{I}_{n}$.
\end{definition}

Let $\overline{L}$ represent the Laplacian matrix of the induced unsigned digraph $\mathcal{\overline G}$. If the signed digraph $\mathcal{ G}$ is structurally balanced, then there exists a gauge transformation $D_n\in \mathcal{D}_n$ such that $\overline{L}=D_nLD_n$ holds.

For any subgraph $\mathcal{G}_s=(\mathcal{V}_s,\mathcal{E}_s,A_s)$ of $\mathcal{G}=(\mathcal{V},\mathcal{E},A)$, the neighbor set $\mathcal{N}_{\mathcal{G}_s}$ of the subgraph $\mathcal{G}_s$ is denoted by
\begin{equation*}
\mathcal{N}_{\mathcal{G}_s}=\{v_j:(v_j,v_i)\in \mathcal{E},~\forall v_i\in \mathcal{V}_s,\forall v_j\in \mathcal{V}\backslash\mathcal{V}_s\}.
\end{equation*}
where $\mathcal{V}\backslash\mathcal{V}_s=\{v_j:v_j\in \mathcal{V}~\mbox{but}~ v_j\notin \mathcal{V}_s\}$. When considering a signed digraph $\mathcal{G}$ under arbitrary topology, a node $v_i$ is called a leader if $v_i$ is included in some strongly connected subgraph $\mathcal{G}_s$ of $\mathcal{G}$ that satisfies $\mathcal{N}_{\mathcal{G}_s}=\varnothing$; otherwise, the node $v_i$ is called a follower. The sets of leaders and followers of $\mathcal{G}$ are denoted by $\mathcal{L}$ and $\mathcal{F}$, respectively.


\section{Problem Formulation}\label{Problem Formulation}

Consider signed networks with $n$ nodes given by $\mathcal{V}=\{v_i: i\in\mathcal{I}_n\}$. For every node $v_i$, its single-integrator dynamics  can be described by
\begin{equation}\label{eq4}
\dot{x}_i(t)=u_i(t),~~i\in\mathcal{I}_n
\end{equation}

\noindent where $x_i(t)\in\mathbb{R}$ and $u_i(t)\in\mathbb{R}$ are the state and the control input of the node $v_i$, respectively. From \cite{Alta13}, the control input $u_i(t)$ is provided by
\begin{equation}\label{eq5}
u_i(t)=\sum_{j\in\mathcal{N}_i}a_{ij}[x_j(t)-\mathrm{sgn}(a_{ij})x_i(t)],~~i\in\mathcal{I}_n.
\end{equation}
With $L$, we can write (\ref{eq4}) and (\ref{eq5}) in a compact form as follows
\begin{equation}\label{eq6}
\dot{\bm x}(t)=-L\bm x(t)
\end{equation}
where $\bm{x}(t)=[x_1(t), x_2(t), \dots, x_n(t)]^{\top}\in \mathbb{R}^{n}$. Denote $\theta_i\triangleq\lim_{t\rightarrow \infty}x_i(t)$, $\forall i\in \mathcal{I}_n$. For arbitrary initial state $x_i(0)$, $\forall i\in \mathcal{I}_n$, the system (\ref{eq6})  reaches
\begin{enumerate}
\item[1)] bipartite consensus if $|\theta_i|=|\theta_j|$,~~$\forall i,j\in \mathcal{I}_n$;
\item[2)] interval bipartite consensus if
\begin{equation*}
\theta_i\left\{
\begin{aligned}
&\in \{-\overline{\theta},~\overline{\theta}\},~~v_i~\mbox{is a rooted node}\\
&\in [-\overline{\theta},~\overline{\theta}],~~~v_i~\mbox{is a non-rooted node}\\
\end{aligned},~~\forall i\in \mathcal{I}_n;
\right.
\end{equation*}
\item[3)] bipartite containment tracking if
\begin{equation*}
\theta_i\in \cup_{v_j\in \mathcal{L}}[-|\theta_j|,|\theta_j|],~~\forall i\in \mathcal{I}_n;
\end{equation*}
\item[4)] state stability if $\theta_i=0,~~\forall i\in \mathcal{I}_n$.
\end{enumerate}

When the signed digraph $\mathcal{G}$ contains a spanning tree and satisfies the condition C1) or C2), then it follows from Lemma \ref{lem1} that the Laplacian matrix $L$ has one zero eigenvalue and $n-1$ nonzero eigenvalues with positive real parts. Let  $\eta$ and $\xi$  be the left and right eigenvector of $L$ associated with the zero eigenvalue, respectively. 
Taking advantage of \cite[Lemma 5.1]{Du19}, the terminal value of the system (\ref{eq6}) can be calculated by
\begin{equation}\label{eq7}
\lim_{t\rightarrow \infty}\bm x(t)=\frac{\xi\eta^{\top}}{\eta^{\top}\xi}\bm x(0)
\end{equation}

\noindent where $\bm{x}(0)=[x_1(0),x_2(0),\cdots,x_n(0)]^{\top}\in \mathbb{R}^{n\times 1}$ represents the initial states of all nodes. 
From (\ref{eq7}), we  know that the right eigenvector $\xi$ and the left eigenvector $\eta$ play an important role in investigating the collective behaviors of the system (\ref{eq6}) when the quasi-strongly connected signed digraph $\mathcal{G}$ satisfies the condition C1) or C2). 
Based on \cite{Du19}, the  mathematical expression of the left eigenvector $\eta$ is given by
\begin{equation*}
\eta=[\eta^{\top}_{\mathrm{r}},\eta^{\top}_{\mathrm{nr}}]^{\top}
\end{equation*}

\noindent with $\eta_{\mathrm{r}}=[\det(L_{11}),\det(L_{22}),\cdots,\det(L_{mm})]^{\top}$ and $\eta_{\mathrm{nr}}=\bm{0}_{n-m}$. 
It can be obtained from (\ref{eq7}) that the right eigenvector $\xi$ determines the dynamic behaviors of the system (\ref{eq6}). 
However,  mathematical expression for the right eigenvector $\xi$ has not been developed yet.

In the following, we are interested in deriving the mathematical expression of $\xi$, based on which the algebraic criteria for collective behaviors of signed networks under arbitrary communication topologies are developed.

\section{Mathematical Expression of Right Eigenvector}

In this section, for quasi-strongly connected signed digraphs, we aim at providing the mathematical expression for the right eigenvector $\xi$ of the Laplacian matrix $L$ associated with the zero eigenvalue. Without loss of generality, suppose that the quasi-strongly connected signed digraph $\mathcal{G}$ contains $m (1\leq m\leq n)$ rooted nodes and $n-m$ non-rooted nodes, and its Laplacian matrix $L$ is given by (\ref{eq1}).

When the signed digraph $\mathcal{G}$ meets the condition C1) or C2), we know that the subgraph $\mathcal{G}_\mathrm{r}$ is structurally balanced. Thus, there exists a gauge transformation $D_m\in \mathcal{D}_m$ such that $\overline{L}_\mathrm{r}=D_mL_\mathrm{r}D_m$, where $\overline{L}_\mathrm{r}$ is the Laplacian matrix of the induced unsigned digraph $\overline{\mathcal{G}}_\mathrm{r}$ corresponding to $\mathcal{G}_\mathrm{r}$. Motivated by $L$, we can give a constructive theorem for the right eigenvector $\xi$.

\begin{theorem}\label{thm1}
Consider a signed digraph $\mathcal{G}$ that is quasi-strongly connected and includes $m$ rooted nodes. 
Its Laplacian matrix $L$ of $\mathcal{G}$ is given by (\ref{eq1}). 
If $\mathcal{G}$ satisfies the condition C1) or C2), then the right eigenvector $\xi$ is constructed by
\begin{equation}\label{equ9}
\xi=\left[
  \begin{array}{c}
  \xi_{\mathrm{r}}\\
  \xi_{\mathrm{nr}}
  \end{array}
\right]
\end{equation}
with $\xi_{\mathrm{r}}=D_m1_m$ and $\xi_{\mathrm{nr}}=\left[\xi_{m+1}\right.$, $\xi_{m+2}$, $\cdots$, $\left.\xi_{n}\right]^{\top}$ whose entry satisfies
\begin{equation}\label{equ8}
\xi_{m+j}=-\frac{\mathrm{det}(\Phi_j)}{\mathrm{det}(L_{\mathrm{nr}}+B)},~\forall j\in\{1,2,\cdots,n-m\}
\end{equation}
where the matrix $\Phi_j \in \mathbb{R}^{(n-m)\times (n-m)}$ is derived by substituting the column $j$ of $L_{\mathrm{nr}}+B$ with the vector $-A_{\mathrm{rnr}}\xi_\mathrm{r}=[e_1$, $e_2$, $\cdots$, $e_{n-m}]^{\top}\in \mathbb{R}^{n-m}$.
\end{theorem}

\begin{proof}
Since the signed digraph $\mathcal{G}$ meets the condition C1) or C2), we  know that the subgraph $\mathcal{G}_\mathrm{r}$ is strongly connected and structurally balanced. 
There exists a gauge transformation $D_m\in \mathcal{D}$ such that $\overline{L}_{\mathrm{r}}=D_mL_\mathrm{r}D_m$ and $L_\mathrm{r}D_m\bm{1}_m=\bm0_m$ hold. 
Moreover, the Laplacian matrix $L_\mathrm{r}$ has a zero eigenvalue and $m-1$ non-zero eigenvalues with positive reals. 
Thus,  $\bm{1}_{m}$ is the right eigenvector of $\overline{L}_\mathrm{r}$ associated with the zero eigenvalue. 
This, together with $\overline{L}_{\mathrm{r}}=D_mL_\mathrm{r}D_m$, ensures that $\xi_\mathrm{r}=D_m\bm{1}_m$ is the right eigenvector of $L_\mathrm{r}$ corresponding to the zero eigenvalue.

There  exists an inverse matrix $P\in \mathbb{R}^{n\times n}$ such that
\begin{align*}
PL&=\begin{bmatrix}
I_m & \mathbf{0}_{m\times(n-m)}\\
\mathbf{0}_{(n-m)\times m} &  (L_{\mathrm{nr}}+B)^{-1}
\end{bmatrix}
\begin{bmatrix}
L_\mathrm{r} & \mathbf{0}_{m\times (n-m)}\\
-A_{\mathrm{rnr}} & L_{\mathrm{nr}}+B
\end{bmatrix}\\
&=\begin{bmatrix}
L_\mathrm{r} & \mathbf{0}_{m\times (n-m)}\\
-(L_{\mathrm{nr}}+B)^{-1}A_{\mathrm{rnr}} & I_{n-m}
\end{bmatrix}=L'.
\end{align*}

\noindent Because of $L\xi=\bm 0_n$, we can further deduce
\begin{align}\label{equ10}
L'\xi=\bm 0_n.
\end{align}

\noindent It follows from \cite[Lemma 5]{Meng16} that all eigenvalues of $L_{\mathrm{nr}}+B$ has positive real parts. With (\ref{equ10}), it is immediate to obtain
\begin{equation}
\xi_{\mathrm{nr}}=(L_{\mathrm{nr}}+B)^{-1}A_{\mathrm{rnr}}D_m1_m.
\end{equation}

Next,  the expression of $\xi_{\mathrm{nr}}$ is explored. We can calculate
\begin{align*}
(L_{\mathrm{nr}}+B)^{-1}A_{\mathrm{rnr}}&=\frac{1}{\mathrm{det}(L_{\mathrm{nr}}+B)}(L_{\mathrm{nr}}+B)^{*}A_{\mathrm{rnr}}\\
&\triangleq\frac{1}{\mathrm{det}(L_{\mathrm{nr}}+B)}F
\end{align*}

\noindent where
\begin{align*}
(L_{\mathrm{nr}}+B)^*
=\begin{bmatrix}
l_{m+1, m+1}^* & -a_{m+2,m+1}^*  & \dots &-a_{n, m+1}^*\\
-a_{m+1,m+2}^* & l_{m+2,m+2}^* & \dots & -a_{n,m+2}^*\\
\vdots & \vdots   & \ddots & \vdots\\
-a_{m+1,n}^* & -a_{m+2,n}^* & \dots & a_{nn}^*
\end{bmatrix}
\end{align*}

\noindent with the element $-a_{m+i,m+j}^*=(-1)^{(i+j)}\det((L_\mathrm{nr}+B)_{ij})$, $\forall i,j\in \mathcal{I}_{n-m}$ and $l_{m+i, m+i}=(-1)^{(i+i)}\det((L_\mathrm{nr}+B)_{ii})$. $(L_\mathrm{nr}+B)_{ij}$ is developed from $L_\mathrm{nr}+B$ by deleting $i$th row and $j$th column. 

Denote $A_{\mathrm{rnr}}^{(i,:)}$, $\forall i\in \mathcal{I}_{n-m}$ (respectively, $F^{(j,:)}$, $\forall j\in \mathcal{I}_{m}$) as the $i$th (respectively, $j$th) row of the matrix $A_{\mathrm{rnr}}$ (respectively, $F$).  Then, we have
\begin{align}\label{equ12}
F^{(j, :)}=&\{(-a_{m+1,m+j}^*)A_{\mathrm{rnr}}^{(1,:)}+\nonumber(-a_{m+2,m+j}^*)A_{\mathrm{rnr}}^{(2,:)}\\
&\nonumber+\dots+(-a_{n, m+j}^*)A_{\mathrm{rnr}}^{(n-m,:)} \}\\
=\nonumber&\{ (-1)^{1+j}\mathrm{det}((L_{\mathrm{nr}}+B)_{1j})A_{\mathrm{rnr}}^{(1,:)}\\
&+\nonumber(-1)^{2+j}\mathrm{det}((L_{\mathrm{nr}}+B)_{2j})A_{\mathrm{rnr}}^{(2,:)} +\dots\nonumber\\
&+(-1)^{n-m+j}\mathrm{det}((L_{\mathrm{nr}}+B)_{(n-m)j}) A_{\mathrm{rnr}}^{(n-m,:)}\}   .
\end{align}

Let $-A_{\mathrm{rnr}}\xi_\mathrm{r}=[e_1, e_2, \dots, e_{n-m}]^{\top}\in \mathbb{R}^{n-m}$. 
A new matrix $\Phi_j$ can be constructed by replacing the $j$th ($\forall j\in \mathcal{I}_{n-m}$) column of $L_{\mathrm{nr}}+B$ with $-A_{\mathrm{rnr}}\xi_\mathrm{r}$. 
With (\ref{equ12}), the following equation can be derived
\begin{align}\label{expression of delta j}
\mathrm{det}(\Phi_j)
=&\{(-1)^{1+j}e_1 \mathrm{det}((L_{\mathrm{nr}}+B)_{1j})\nonumber\\
&+(-1)^{2+j}e_2\mathrm{det}((L_{\mathrm{nr}}+B)_{2j})+\dots\nonumber\\
&+(-1)^{n-m+j}e_{n-m}\mathrm{det}((L_{\mathrm{nr}}+B)_{(n-m)j})\}\nonumber\\
=&-\{(-1)^{1+j}\mathrm{det}((L_{\mathrm{nr}}+B)_{1j})A_{\mathrm{rnr}}^{(1, :)}\xi_\mathrm{r}\nonumber\\
&+(-1)^{2+j}\mathrm{det}((L_{\mathrm{nr}}+B)_{2j})A_{\mathrm{rnr}}^{(2, :)}\xi_\mathrm{r}+\dots \nonumber\\
&+(-1)^{n-m+j}\mathrm{det}((L_{\mathrm{nr}}+B)_{(n-m)j})A_{\mathrm{rnr}}^{(n-m, :)}\xi_\mathrm{r}\}\nonumber\\
=&-F^{(j,:)}\xi_\mathrm{r}.
\end{align}

By the fact
\begin{equation*}
\begin{split}
\xi_\mathrm{nr}&=-\frac{1}{\mathrm{det}(L_{\mathrm{nr}}+B)}\\
&~~~~~~\times\left[F^{(1,:)}\xi_\mathrm{r},F^{(2,:)}\xi_\mathrm{r},\cdots,F^{(n-m,:)}\xi_\mathrm{r}\right]^{\top},
\end{split}
\end{equation*}

\noindent we  obtain that $(\ref{equ8})$ holds. Therefore, the right eigenvector $\xi$ satisfies (\ref{equ9}). The proof is  completed.
\end{proof}

\begin{remark}
Theorem \ref{thm1} provides  an approach to calculating the right eigenvector $\xi$ by exploiting the Laplacian matrix $L$. 
In comparison with the reference \cite{Du19},   the mathematical expression for the right eigenvector $\xi$ is developed, which is convenient to explore the collective behaviors of signed networks (see Section V for more details).
\end{remark}

\section{Applications in Analyzing Collective Behaviors of Signed Networks}

In this section, we aim at exploring some potential applications for the right eigenvector $\xi$. Based on $\xi$,  the terminal states of the agents in signed networks can be calculated.

\begin{theorem}
Consider the system (\ref{eq6}) under a communication topology that is a quasi-strongly connected signed digraph $\mathcal{G}$ with $n$ rooted nodes. The Laplacian matrix $L$ of $\mathcal{G}$ is given by (\ref{eq1}). Then, there exist the following results.
\begin{enumerate}
  \item If $\mathcal{G}$ satisfies the condition C1) or C2), then the terminal state of the system (\ref{eq6}) is given by
  \begin{equation}\label{equa12}
  \begin{split}
  \lim_{t\rightarrow \infty}x(t)=\left[
  \begin{array}{c}
\Ds{  \frac{\sum_{i=1}^m\sigma_i\det(L_{ii})x_i(0)}{\sum_{i=1}^m\det(L_{ii})}\sigma_1}\\
 \Ds{ \frac{\sum_{i=1}^m\sigma_i\det(L_{ii})x_i(0)}{\sum_{i=1}^m\det(L_{ii})}\sigma_2}\\
  \vdots\\
 \Ds{ \frac{\sum_{i=1}^m\sigma_i\det(L_{ii})x_i(0)}{\sum_{i=1}^m\det(L_{ii})}\sigma_m}\\
\Ds{  -\frac{\sum_{i=1}^m\sigma_i\det(L_{ii})x_i(0)}{\sum_{i=1}^m\det(L_{ii})}\frac{\mathrm{det}(\Phi_1)}{\mathrm{det}(L_{\mathrm{nr}}+B)}}\\
\Ds{  -\frac{\sum_{i=1}^m\sigma_i\det(L_{ii})x_i(0)}{\sum_{i=1}^m\det(L_{ii})}\frac{\mathrm{det}(\Phi_2)}{\mathrm{det}(L_{\mathrm{nr}}+B)}}\\
  \vdots\\
\Ds{  -\frac{\sum_{i=1}^m\sigma_i\det(L_{ii})x_i(0)}{\sum_{i=1}^m\det(L_{ii})}\frac{\mathrm{det}(\Phi_{n-m})}{\mathrm{det}(L_{\mathrm{nr}}+B)}}\\
  \end{array}
  \right].
  \end{split}
  \end{equation}
  \item If $\mathcal{G}$ satisfies the condition C3), then the system (\ref{eq6})  achieves the state stability.
\end{enumerate}
\end{theorem}

\begin{proof}
1): Because the signed digraph $\mathcal{G}$ meets the condition C1) or C2), it follows from Theorem \ref{thm1} and (\ref{eq6}) that (\ref{equa12}) holds.

2): Since the signed digraph $\mathcal{G}$ satisfies the condition C3), from Lemma \ref{lem1}, all eigenvalues of $L$ have positive real parts. Hence, the system (\ref{eq6})  achieves the state stability. The proof is completed.
\end{proof}

From \cite{Meng20}, we know that the structurally balanced nodes play an important role in exploring the interval bipartite consensus of signed networks, based on which the impact index can be developed. 
However,   \cite{Meng20} does not provide an approach to seek structurally balanced (or unbalanced) nodes from all nodes of signed networks. 
Benefitting from $\xi$, the algebraic criteria to identify whether a node is a structurally balanced node or not are provided in the following theorem.

\begin{theorem}\label{thm3}
For a quasi-strongly connected signed digraph $\mathcal{G}$ that satisfies the condition C1) or C2), there exist the following two results.
\begin{enumerate}
  \item The node $v_i$, $\forall i\in \mathcal{I}_n$ is  a structurally balanced node if and only if $|\xi_i|=1$ holds.
  \item The node $v_i$, $\forall i\in \mathcal{I}_n$ is a structurally unbalanced node if and only if $|\xi_i|<1$ holds.
\end{enumerate}
\end{theorem}

\begin{proof}
Since $\mathcal{G}$ satisfies the condition C1) or C2),   the right eigenvector $\xi$ of $L$ associated with the zero eigenvalue is developed from Theorem \ref{thm1}. This, together with \cite[Lemma 5]{Meng20}, ensures  the algebraic criteria for structurally balanced nodes and unbalanced nodes.
\end{proof}

With the help of $\xi$,  the algebraic criteria for collective behaviours of signed networks under quasi-strongly connected signed digraphs can be developed.

\begin{theorem}\label{thm4}
Consider a signed digraph $\mathcal{G}$ that is quasi-strongly connected. Then for arbitrary initial state $\bm{x}(0)$, there exist the following results for the system (\ref{eq6}).
\begin{enumerate}
  \item[1)] The system (\ref{eq6}) reaches the interval bipartite consensus if and only if $\det(L)=0$ holds. 
In particular, the system (\ref{eq6})  reaches the bipartite consensus if and only if $\det(L)=0$ and $|\xi_1|=|\xi_2|=\cdots=|\xi_n|=1$ hold.
  \item[2)] The system (\ref{eq6})  reaches the state stability if and only if $\det(L)\neq0$ holds.
\end{enumerate}
\end{theorem}

\begin{proof}
1) ``\emph{Sufficiency}'': Since $\mathcal{G}$ is quasi-strongly connected and $\det(L)=0$ holds, it follows from Lemma \ref{lem1} that $\mathcal{G}$ satisfies the condition C1) or C2). 
Based on (\ref{eq7}), it directly follows from Theorems \ref{thm1} and \ref{thm3} that the system (\ref{eq6})  achieves the interval bipartite consensus. 
Due to $|\xi_1|=|\xi_2|=\cdots=|\xi_n|=1$, it can be obtained from (\ref{eq7}) that the system (\ref{eq6}) achieves the bipartite consensus.

2) ``\emph{Sufficiency}'': Because $\mathcal{G}$ is quasi-strongly connected and $\det(L)\neq0$ holds, it is immediate to develop that $\mathcal{G}$ satisfies the condition C3) and all eigenvalues of $L$ have positive real parts from Lemma \ref{lem1}. Therefore, the system (\ref{eq6})  achieves the state stability.

1) and 2) ``\emph{Necessity}'': The necessity of 1) and 2) can be induced by the mutually exclusive relation between the structural balance and unbalance of the signed digraph $\mathcal{G}$.
\end{proof}

\begin{remark}
Based on Theorem \ref{thm1},  algebraic criteria for  collective behaviors of signed networks are proposed. 
As we know, unsigned networks can be considered as  particular cases of signed networks when there exist no antagonistic interactions among agents. Therefore, Theorem \ref{thm1} can  be applied to identify the collective behaviors of unsigned networks.
\end{remark}

Based on Theorem \ref{thm3},  an algebraic criterion for the interval bipartite consensus of signed networks in the following corollary is shown.

\begin{corollary}\label{cor2}
For a quasi-strongly connected signed digraph $\mathcal{G}$, let $\det(L)=0$ hold. 
The system (\ref{eq6})  achieves the interval bipartite consensus instead of bipartite consensus if and only if the signed digraph $\mathcal{G}$ contains structurally unbalanced nodes.
\end{corollary}

\begin{proof}
A direct consequence of (\ref{eq7}) and Theorem \ref{thm3}.
\end{proof}

When the signed digraph $\mathcal{G}$ is not connected, it can be divided into several connected subgraphs, each of which has a spanning tree and has no connections with  other subgraphs. 
Without loss of generality, we suppose that there exist $k$ $(k\geq 2)$ connected subgraphs $\mathcal{G}_{1}^s$, $\mathcal{G}_{2}^s$, $\cdots$, $\mathcal{G}_{k}^s$ in the signed digraph $\mathcal{G}$. Let $L_{s1}\in \mathbb{R}^{s_1\times s_1}$, $L_{s2}\in \mathbb{R}^{s_2\times s_2}$, $\cdots$, and $L_{sk}\in\mathbb{R}^{s_k\times s_k}$ denote the Laplacian matrices of the subgraphs $\mathcal{G}_{1}^s$, $\mathcal{G}_{1}^s$, $\cdots$, and $\mathcal{G}_{k}^s$, respectively, where $s_1+s_2+\cdots+s_k=n$. Besides, the Laplacian matrix $L$ of $\mathcal{G}$ can be written as
\begin{equation}\label{equ13}
L=\left[
  \begin{array}{cccc}
    L_{s1} &  &  &\\
     & L_{s2} &  &\\
     & & \ddots &\\
     & & & L_{sk}
  \end{array}
\right]_{n\times n}.
\end{equation}

The following theorem gives an algebraic criterion for collective behaviors of signed networks whose communication topologies are arbitrary signed digraphs.

\begin{theorem}\label{thm5}
Consider a signed digraph $\mathcal{G}$ under arbitrary topology. Then, the system (\ref{eq6})  achieves
\begin{enumerate}
  \item the bipartite containment tracking if and only if at least one node of $\mathcal{G}$ is a structurally balanced node;
  \item the state stability if and only if all nodes of $\mathcal{G}$ are structurally unbalanced nodes.
\end{enumerate}
\end{theorem}

\begin{proof}
``\emph{Sufficiency}'' 1): From (\ref{equ13}), the system (\ref{eq6}) can be devided into $k$ subsystems as follows
\begin{equation}\label{equ14}
\dot{x}^{sj}=-L_{sj}x^{sj},~\forall j\in \mathcal{I}_k
\end{equation}

\noindent where $x^{sj}\in \mathbb{R}^{s_j}$ denotes the state vector. All subgraphs $\mathcal{G}_{1}^s$, $\mathcal{G}_{2}^s$, $\cdots$, $\mathcal{G}_{k}^s$ of $\mathcal{G}$ have a spanning tree. Let $\mathcal{B}_i$ (respectively, $\mathcal{U}_i$) represent the set of structurally balanced (respectively, unbalanced) nodes of $\mathcal{G}_i$. For the subgraph $\mathcal{G}_j$, $\forall j\in \mathcal{I}_k$, there exist the following two cases.

\emph{Case 1)}. When the subgraph $\mathcal{G}_j$ consists of structurally balanced nodes, it can be obtained that all rooted nodes of $\mathcal{G}_i$ are structurally balanced nodes, which leads to that $\mathcal{G}_j$ satisfies the condition C1) or C2). 
Based on (\ref{eq7}) and Theorem \ref{thm4}, we  obtain that the convergence behaviors of the subsystem (\ref{equ14}) are as follows
\begin{equation}\label{equ15}
\theta_i\in \cup_{v_p\in \mathcal{B}_j\cap \mathcal{L}}[-|\theta_p|,|\theta_p|],~i\in\{1,2,\cdots,s_j\}.
\end{equation}

\emph{Case 2)}. When the subgraph $\mathcal{G}_j$ has no structurally balanced nodes, it can be developed that all nodes of $\mathcal{G}_i$ are structurally unbalanced nodes, which implies that $\mathcal{G}$ satisfies the condition C3). Since all eigenvalues of $L_{sj}$ have positive real parts, we  derive that
\begin{equation}\label{equ16}
\theta_i\in \cup_{v_p\in \mathcal{U}_j\cap\mathcal{L}}[-|\theta_p|,|\theta_p|]={0},~i\in\{1,2,\cdots,s_j\}.
\end{equation}

Based on (\ref{equ15}) and (\ref{equ16}), we  further deduce
\begin{equation*}
\theta_i\in \cup_{v_p\in \mathcal{L}}[-|\theta_p|,|\theta_p|],~~\forall i\in \mathcal{I}_n
\end{equation*}

\noindent which indicates that the system (\ref{eq6})  achieves the bipartite containment tracking.

2): Since all nodes of $\mathcal{G}$ are structurally unbalanced nodes,  all eigenvalues of $L$ have positive real parts. Hence, the system (\ref{eq6})  reaches the state stability.

``\emph{Necessity}'' 1) and 2): The necessity of 1) and 2) can be deduced by the mutually exclusive relation between  structurally balanced nodes and unbalanced nodes of the signed digraph $\mathcal{G}$. The proof is completed.
\end{proof}

\section{Simulation Results}
In this section, we introduce one example to demonstrate the effectiveness of our developed theoretical results of signed networks described by the system (\ref{eq6}). Without loss of generality, the initial state $\bm{x}(0)$ of the system (\ref{eq6}) is selected as
\begin{equation}
\bm{x}(0)=[1, 2, -1, 2, 1, -2, 1, -1, 1, -1, -3, -2, 2]^{\top}.
\end{equation}

\begin{figure}[!htbp]
\centering
\includegraphics[width=8cm]{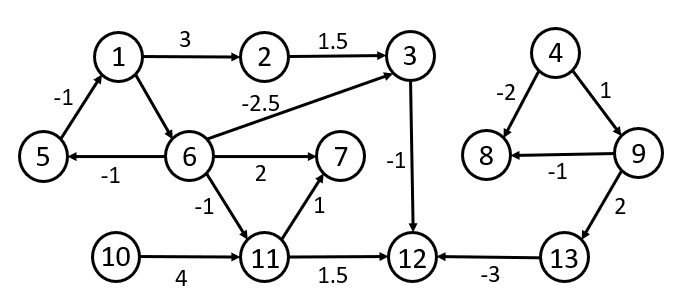}
\caption{Signed digraph $\mathcal{G}_a$.}
\label{tp1}
\end{figure}

\begin{example}
We use the signed digraph $\mathcal{G}_a$ in Fig. \ref{tp1} to denote the communication topology of the system (\ref{eq6}). 
From Fig. \ref{tp1},  the signed digraph $\mathcal{G}_a$ contains a spanning tree and its Laplacian matrix $L_a$ is provided by
\begin{align*}
L_a=
\begin{bmatrix}
\begin{smallmatrix}
     1  &0     & 0  &0  &1  &0  &0  &0  &0  &0  &0  &0  &0 \\
    -3  &3     &0   &0  &0  &0  &0  &0  &0  &0  &0 & 0  &0 \\
     0  &-1.5  &4  &0  &0 &2.5 &0  &0  &0  &0  &0  &0  &0 \\
     0  &0     &0   &0  &0  &0  &0  &0  &0  &0  &0  &0  &0 \\
     0  &0     &0   &0  &1  &1  &0  &0  &0  &0  &0  &0  &0 \\
    -2  &0     &0  &0  &0  &2  &0  &0  &0  &0  &0  &0  &0 \\
     0  &0     &-1  &0  &0 &-2  &4  &0  &0  &0 &-1  &0  &0 \\
     0  &0     &0  &2  &0  &0  &0  &3  &1  &0  &0  &0  &0 \\
     0  &0     &0 &-1  &0  &0  &0  &0  &1  &0  &0  &0  &0 \\
     0  &0     &0  &0  &0  &0  &0  &0  &0  &0  &0  &0  &0 \\
     0  &0     &0  &0  &0  &1  &0  &0  &0 &-4  &5  &0  &0 \\
     0  &0     &1  &0  &0  &0  &0  &0  &0  &0 &-1.5 &5.5 &3 \\
     0  &0     &0  &0  &0  &0  &0  &0 &-2  &0  &0  &0  &2
\end{smallmatrix}
\end{bmatrix}.
\end{align*}

The determinant of $L_a$ is $\det(L_a)=0$. According to $L_a$ and Theorem \ref{thm1}, we get the right eigenvector of $L_a$ associated with the zero eigenvalue as follows
\begin{equation*}
\xi^{a}=[-1,-1, \frac{1}{4}, 1, 1,-1,-\frac{3}{16}, -1, 1, 1, 1, -\frac{7}{22}, 1]^{\top}.
\end{equation*}
It follows from Theorem \ref{thm3} that the nodes $v_1$, $v_2$,  $v_4$, $v_5$, $v_6$, $v_8$, $v_9$, $v_{10}$, $v_{11}$, $v_{13}$, are structurally balanced nodes. 
The nodes $v_3$, $v_7$, $v_{12}$ are structurally unbalanced nodes.

\begin{figure}[!htbp]
\centering
\includegraphics[scale=.23]{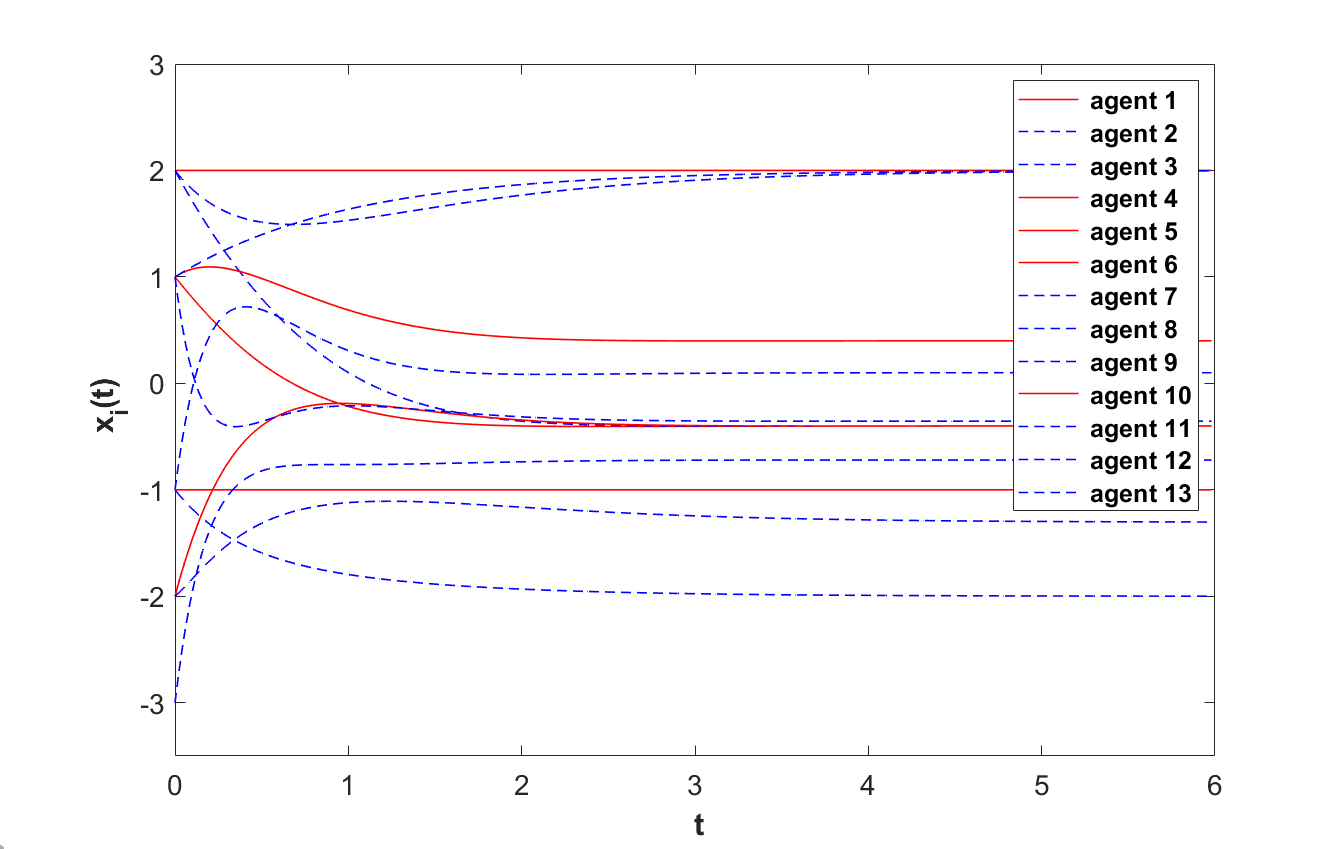}
\caption{State evolution of the system (\ref{eq6}) under the signed digraph $\mathcal{G}_a$.}
\label{sim1}
\end{figure}

The state evolution of the system (\ref{eq6}) is plotted in Fig. \ref{sim1}. This figure  depicts that the system (\ref{eq6}) achieves bipartite containment tracking, which is consistent with the developed results of Theorem \ref{thm5}.
\end{example}

\section{Conclusions}

In this paper, we have explored  algebraic criteria to identify the collective behaviors of signed networks. 
Toward this end,  an expression for the right eigenvector of Laplacian matrix associated with the zero eigenvalue has been given. 
Benefitting from the right eigenvector, we have developed  methods to distinguish structurally balanced nodes and unbalanced nodes from all nodes of signed digraphs, and to determine the bipartite consensus, interval bipartite consensus and state stability of signed networks. 
Moreover,  algebraic criteria for the bipartite containment tracking of signed networks have been proposed. 
In addition, we have introduced one simulation example to illustrate the correctness of our developed results.

\end{document}